\documentclass{article}
\usepackage{maa-monthly}

\theoremstyle{theorem}
\newtheorem{theorem}{Theorem}

\theoremstyle{definition}

\begin{document}

\title{Squares in arithmetic progressions and  infinitely many primes}
\markright{Primes and Progressions}
\author{Andrew Granville}

% \thanks{MSC codes: Primary  11N05; Secondary     11A41,  11B25. }   

\maketitle

\begin{abstract}
We give a new proof that there are infinitely many primes, relying on  van der Waerden's theorem for coloring the integers, and Fermat's theorem that there cannot be four squares in an arithmetic progression. We go on to discuss where else these ideas have come together in the past.
\end{abstract}

\section{Infinitely many primes}

Levent Alpoge  recently gave a rather different proof \cite{Alp} that there are infinitely many primes. His starting point was the famous result of van der Waerden (see, e.g., \cite{TV}):
\medskip 

\noindent \textbf{van der Waerden's Theorem.} \emph{Fix integers $m\geq 2$ and $\ell\geq 3$. If  every positive integer is assigned one of $m$ colors, in any way at all, then there is an $\ell$-term arithmetic progression  of integers which each have the same color.}
\medskip

Using a clever coloring in van der Waerden's theorem, and some elementary number theory, Alpoge deduced that there are infinitely many primes. We proceed from van der Waerden's theorem a little differently, employing a famous result of Fermat (see, e.g., \cite{Sil}):

\medskip

\noindent \textbf{Fermat's Theorem.} \emph{There are no four-term arithmetic progressions of distinct integer squares.}
\medskip

From these two results we deduce the following:

\begin{theorem}  There are infinitely many primes.
\end{theorem}

\begin{proof}
If there are  only finitely many primes $p_1,\ldots, p_k$, then every integer $n$ can be written as $p_1^{e_1}\cdots p_k^{e_k}$ for some integers $e_1,e_2,\ldots,e_k\geq 0$. We can write each of these exponents $e_j$ as
\[
e_j=2q_j+r_j, \ \text{where} \ r_j  \ \text{  is the ``remainder''  when dividing} \ e_j \ \text{by 2,}
\]
and equals 0 or 1. Therefore if we let
\[
R=p_1^{r_1}\cdots p_k^{r_k}   
\]
then $R$ is a squarefree integer that divides $n$, and 
\[
n/R \  \text{is the square of an integer}.
\]
(In fact, $n/R=Q^2$ where $Q=p_1^{q_1}\cdots p_k^{q_k}$.)

We will use $2^k$ colors to color the integers: Integer $n$ is colored by the vector $(r_1,\ldots,r_k)$. 
By van der Waerden's theorem there are four integers in arithmetic progression
\[ A, A+D, A+2D, A+3D, \ \text{with} \ D\geq 1,\] 
which all have the same color $(r_1,\ldots,r_k)$. Now $R=p_1^{r_1}\cdots p_k^{r_k}$ divides each of these numbers, so also divides $D=(A+D)-A$. Letting
$a=A/R$ and $d=D/R$, we see that 
\[ a,a+d,a+2d,a+3d \ \text{are four squares in arithmetic progression,}  \] 
contradicting Fermat's theorem.
\end{proof}

These ideas have come together before to make a rather different, not-too-obvious deduction:

\section{The number of squares in a long arithmetic progression}

Let $Q(N)$ denote the maximum number of squares that there can be in an arithmetic progression of length $N$.
A slight refinement of the Erd\H os--Rudin conjecture states that the maximum number is attained by the arithmetic progression
\[
\{ 24n+1:\ 0\leq n\leq N-1 \}
\]
which contains $\sqrt{ \frac 83 N}$ squares, plus or minus one. From Fermat's theorem one easily sees that  
$$
Q(N) \leq \frac{3N+3} 4,
$$
but it is difficult to see how to improve the bound to, say, $Q(N)\leq \delta N+b$ for some constant $\delta <\frac 34$.

 It was this problem that inspired one of the most influential results \cite{Sz2}  in combinatorics and analysis (see, e.g., \cite{Gow}):
 \medskip

\noindent \textbf{Szemer\' edi's Theorem.} \emph{Fix $\delta>0$ and integer $\ell\geq 3$. If $N$ is sufficiently large (depending on $\delta$ and $\ell$) then any subset $A$ of $\{ 1,2,\ldots,N\}$ with $\geq \delta N$ elements, must contain an $\ell$-term arithmetic progression.}
\medskip

 van der Waerden's theorem is a consequence of Szemer\' edi's theorem, because if we let $\delta=1/m$ and we color the integers in 
$\{ 1,2,\ldots,N\}$ with $m$ colors, then at least one of the colors is used for at least $N/m$ integers. We apply 
 Szemer\' edi's theorem to this subset $A$ of $\{ 1,2,\ldots,N\}$, to obtain an $\ell$-term arithmetic progression  of integers which each have the same color.
 
 In \cite{Sz1}, Szemer\' edi applied his result to the question of squares in arithmetic progressions:
 
 \begin{theorem} [Szemer\' edi]   For any constant $\delta>0$, if $N$ is sufficiently large, then $Q(N)<\delta N$.
\end{theorem}

\begin{proof} Suppose that there are at least $\delta N$ squares in the arithmetic progression 
$\{ r+ns: \ n=1,2,\ldots ,N\}$ with $s\geq 1$; that is, there exists a subset $A$ of $\{ 1,2,\ldots,N\}$ with at least $\delta N$ elements for which 
\[
r+ns \ \text{is a square, whenever} \ a\in A.
\]
Szemer\' edi's theorem with $\ell=4$ then implies that $A$ contains a four-term arithmetic progression, say
$u+jv$ for $j=0,1,2,3$. For these values of $n$, we have $r+ns=a+jd$, where $a=r+us$ and $d=vs>1$. That is, we have shown that 
\[ a,a+d,a+2d,a+3d \ \text{are four squares in arithmetic progression,}  \] 
contradicting Fermat's theorem.
\end{proof}

\section{More heavy machinery}

One day over lunch, in late 1989, Bombieri showed me 
a completely different proof of Theorem 2, this time relying on  one of the most influential results   in algebraic and arithmetic geometry, Faltings'  theorem \cite{BGU}. Faltings'  theorem is not easy to state, requiring a general understanding of an algebraic curve and its genus. The basic idea is that an equation in two variables with rational coefficients has only finitely many rational solutions (that is, solutions in which the two variables are rational numbers), unless the equation ``boils down to'' an equation of degree 1, 2 or 3.  To be precise about ``boiling down'' involves the concept of {\sl genus}, which is too complicated to explain here (see, e.g., \cite{BGU}). Here we only need a simple consequence of Faltings' theorem.\medskip

\noindent \textbf{Corollary to Faltings' Theorem.} \emph{Let $b_1,b_2,\ldots,b_k$ be distinct integers with $k\geq 5$. Then there are only finitely many rational numbers $x$ for which 
\[
(x+b_1)(x+b_2)\cdots (x+b_k) \ \text{is the square of a rational number}.
\]
}
\medskip

\begin{proof} [Another proof of Theorem 2] Fix an integer $M>6/\delta$.  Let $B(M)$ be the total number of rational numbers $x$
and integer 6-tuples $ b_1=0<  b_2<\ldots < b_6\leq M-1$ for which 
$(x+b_1)(x+b_2)\cdots (x+b_6)$ is the square of a rational number.  Faltings' theorem implies that 
$B(M)$ is some finite number, as there are 
 only  finitely many choices for the $b_j$. We let $N$ be any integer $\geq M(B(M)+5)$.

The interval $[0,N-1]$ is covered by the sub-intervals $ I_j$ for $j=0,1,2,\ldots,k-1$, where
$I_j$ denotes the interval $[jM, (j+1)M)$, and 
$kM$ is the smallest multiple of $M$ that is greater than $N$.

Let $\mathcal N:=\{ n:\ 0\leq n\leq N-1 \ \text{and} \ a+nd  \ \text{ is a square}\}$, where the arithmetic progression is chosen so that $|\mathcal N| =Q(N)$. Let  $\mathcal N_j=\{ n\in \mathcal N:\ n\in I_j\}$ for each integer $j$. Let $J$ be the set of integers $j$ for which $\mathcal N_j$ has six or more elements.

Now if $n_1<n_2<\ldots<n_6$ all belong to $\mathcal N_j$, write  $x=a/d+n_1$ and $b_i=n_i-n_1$ for $i=1,\ldots ,6$, so that    
\[
b_1=0<  b_2<\ldots < b_6\leq M-1
\]
and   each $x+b_i=a/d+n_i=(a+n_id)/d$, which implies that
\[
(x+b_1)(x+b_2)\cdots (x+b_6)  = \frac{(a+n_1d)(a+n_2d)\cdots (a+n_6d)}{d^6} 
\]
\begin{center}
is the square of a rational number.
\end{center}
This gives rise to one of the $B(M)$ solutions counted above, and all the solutions created in this way are distinct (since given 
$x,d,b_1,\ldots,b_6$ we have each $a+n_jd=d(x+b_j)$). Therefore the set $\mathcal N_j$ gives rise to $\binom{|\mathcal N_j|}6$ such solutions,  and so in total we have
\[
\sum_{j\in J} \binom{|\mathcal N_j|}6 \leq B(M).
\]

It is easy to verify that $r\leq 5+\binom r6$ for all integers $r\geq 1$, and so 
\[
Q(N)=|\mathcal N| = \sum_{j=0}^{k-1}|\mathcal N_j| \leq \sum_{j=0}^{k-1} 5 + \sum_{j\in J} \binom{|\mathcal N_j|}6  \leq 5k+B(M),
\]
as $|\mathcal N_j| \leq 5$ if $j\not\in J$. Finally, as $k\leq N/M+1$ we have
\[
Q(N)  \leq 5k+B(M) \leq \frac{5N}M + (B(M)+5) \leq \frac{6N}M < \delta N,
\]
as desired.
 \end{proof}

Bombieri \cite{BGP} went on, together with Granville and Pintz,  to combine these two proofs (along with much more arithmetic geometry machinery), to prove that 
\[
Q(N)<N^c
\]
for any $c>\frac 23$, for sufficiently large $N$. Bombieri and Zannier \cite{BZ} improved this to $c>\frac 35$ with a rather simpler proof.
The conjecture that $Q(N)$ behaves more like  a constant times $N^{1/2}$ remains open.

\begin{biog}
\item[Andrew Granville] 
%(and.granville@gmail.com) was educated, and has been a professor, in the UK, Canada and the US. He is currently writing a number theory trilogy: a first course in number theory, an introduction to Diophantine equations, and an introduction to the distribution of primes.
\begin{affil}
D\'epartement de math\'ematiques et de statistique, Universit\'e de Montr\'eal, CP 6128 succ. Centre-Ville, Montr\'eal, QC H3C 3J7, Canada\\
   andrew@dms.umontreal.ca\\
\text{\rm and} \\
Department of Mathematics,
University College London, 
Gower Street, 
London, WC1E 6BT, 
United Kingdom \\
   a.granville@ucl.ac.uk
   \end{affil}

\end{biog}
\vfill\eject

\end{document}